\documentclass[12pt]{amsart}

\usepackage{amssymb}
\usepackage{amsthm}
\usepackage{amsmath}
\usepackage{nicefrac}

\usepackage[margin=1in]{geometry}

\usepackage{enumitem}

\usepackage{dsfont}

\usepackage{ mathrsfs }

\usepackage[foot]{amsaddr}

\usepackage{hyperref}

\usepackage{mathrsfs}

\theoremstyle{plain}
\newtheorem{proposition}{Proposition}[section]
\newtheorem{theorem}[proposition]{Theorem}
\newtheorem{lemma}[proposition]{Lemma}
\newtheorem{corollary}[proposition]{Corollary}
\theoremstyle{definition}

\newtheorem{definition}[proposition]{Definition}

\theoremstyle{remark}
\newtheorem{remark}[proposition]{Remark}

\DeclareMathOperator{\Euc}{Euc} 
\DeclareMathOperator{\id}{id}

\DeclareMathOperator{\dist}{dist}

\DeclareMathOperator{\Lc}{\mathcal{L}}

\DeclareMathOperator{\Levi}{\mathscr{L}}

\DeclareMathOperator{\Cc}{\mathcal{C}}

\DeclareMathOperator{\Sc}{\mathcal{S}}

\DeclareMathOperator{\Bb}{\mathbb{B}}
\DeclareMathOperator{\Cb}{\mathbb{C}}
\DeclareMathOperator{\Db}{\mathbb{D}}

\DeclareMathOperator{\Nb}{\mathbb{N}}

\DeclareMathOperator{\Rb}{\mathbb{R}}

\DeclareMathOperator{\Bf}{\mathsf{K}}

\newcommand{\abs}[1]{\left|#1\right|}

\newcommand{\norm}[1]{\left\|#1\right\|}

\begin{document}

\title[Hankel operators on domains with bounded intrinsic geometry]{Hankel operators on domains with bounded intrinsic geometry}
\author{Andrew Zimmer}\address{Department of Mathematics, University of Wisconsin, Madison, WI, USA}
\email{amzimmer2@wisc.edu}
\date{\today}
%\subjclass[2020]{32A36, 47B35, 32Q15, 32F45, 53C20}

\begin{abstract} In this paper we consider Hankel operators on domains with bounded intrinsic geometry. For these domains we characterize the $L^2$-symbols where the associated Hankel operator is compact (respectively bounded) on the space of square integrable holomorphic functions. \end{abstract}

\maketitle

\section{Introduction}

Given a bounded domain $\Omega \subset \Cb^d$, let $\mu$ denote the Lebesgue measure and let $L^2(\Omega)=L^2(\Omega,\mu)$.  Also let $A^2(\Omega) \subset L^2(\Omega)$ denote the subspace of holomorphic functions and $P_\Omega : L^2(\Omega) \rightarrow A^2(\Omega)$ denote the \emph{Bergman projection}, i.e. the orthogonal projection of $L^2(\Omega)$ onto $A^2(\Omega)$. Finally, given $\phi \in L^{2}(\Omega)$, the associated \emph{Hankel operator} $H_\phi$ has domain 
\begin{align*}
{\rm dom}(H_\phi) = \left\{ f \in A^2(\Omega) : \phi \cdot f \in L^2(\Omega)\right\}
\end{align*}
and is defined by
\begin{align*}
H_\phi(f) = (\id-P_\Omega)(\phi \cdot f) = \phi \cdot f - P_\Omega(\phi \cdot f). 
\end{align*}

We will let $\Sc(\Omega) \subset L^{2}(\Omega)$ denote the symbols $\phi$ where the associated Hankel operator is densely defined on $A^2(\Omega)$. We always have $L^\infty(\Omega) \subset \Sc(\Omega)$ and when $L^\infty(\Omega) \cap A^2(\Omega)$ is dense in $A^2(\Omega)$ (e.g. $\Omega$ is strongly pseudoconvex~\cite[Theorem 3.1.4]{Catlin1980} or star shaped) we have $L^2(\Omega) = \Sc(\Omega)$.

In this paper we consider the well studied problem of characterizing the symbols with compact Hankel operator. The results of this paper are especially inspired by Li's characterization for strongly pseudoconvex domains~\cite{Li1994} (see~\cite{Li1992,LiLuecking1994,Luecking1992} for closely related results). For such domains, Li proved that a Hankel operator is compact if and only if on each sufficiently small metric ball $\Bb_\Omega(\zeta, r)$ in the Bergman metric the symbol can be approximated by a holomorphic function. More precisely: 

\begin{theorem}[Li~\cite{Li1994}]\label{thm:Lis_theorem} Suppose $\Omega \subset \Cb^d$ is a strongly pseudoconvex domain and $\phi \in L^2(\Omega)$. Then the following are equivalent: 
\begin{enumerate}
\item $H_\phi$ extends to a compact operator on $A^2(\Omega)$,  
\item for some $r > 0$ 
$$\lim_{\zeta \rightarrow \partial \Omega} \inf\left\{\frac{1}{\mu(\Bb_\Omega(\zeta,r))} \int_{\Bb_\Omega(\zeta,r)} \abs{f-h}^2  d\mu : h \in {\rm Hol}\left(\Bb_\Omega(\zeta,r)\right)\right\} = 0.$$ 
\end{enumerate}
\end{theorem}

In this paper we extend Li's result to domains with bounded intrinsic geometry, see Definition~\ref{defn:BBG} below. This class of domains was introduced in~\cite{Z2021} and include many well studied families of domains such as
\begin{enumerate}
\item strongly pseudoconvex domains, 
\item finite type domains in $\Cb^2$, 
\item convex domains or more generally $\Cb$-convex domains which are Kobayashi hyperbolic (with no boundary regularity assumptions), 
\item simply connected domains which have a complete K\"ahler metric with pinched negative sectional curvature, 
\item bounded homogeneous domains, and
\item the Bers embeddings of the Teichm\"uller space  of hyperbolic surfaces of genus $g$ with $n$ punctures.
\end{enumerate} 
Further, by definition, any domain biholomorphic to one of the domains listed above also has bounded intrinsic geometry. 

As in the classical strongly pseudoconvex case, for domains with bounded intrinsic geometry we show that compactness of a Hankel operator is equivalent to the symbol being locally approximable by holomorphic functions in a $L^2$-space, but instead of using a scaled Lebesgue measure we use the Riemannian volume form $dV_\Omega$ induced by the Bergman metric. 

\begin{theorem}\label{thm:main} Suppose $\Omega \subset \Cb^d$ is a bounded domain with bounded intrinsic geometry and $\phi \in \Sc(\Omega)$. Then the following are equivalent: 
\begin{enumerate}
\item $H_\phi$ extends to a compact operator on $A^2(\Omega)$,  
\item for some $r > 0$ 
$$\lim_{\zeta \rightarrow \partial \Omega} \inf\left\{ \int_{\Bb_\Omega(\zeta,r)} \abs{\phi-h}^2  dV_\Omega  : h \in {\rm Hol}\left(\Bb_\Omega(\zeta,r)\right)\right\} = 0.$$ 
\end{enumerate}
\end{theorem}

\begin{remark}
For a strongly pseudoconvex domain $\Omega \subset \Cb^d$, we will show that for any $r > 0$ there exists $C=C(r)>1$ such that: if $\zeta \in \Omega$, then
\begin{align}
\label{eqn:volume_equivalence}
\frac{1}{C} dV_\Omega \leq \frac{1}{\mu(\Bb_\Omega(\zeta,r))} d\mu \leq C dV_\Omega
\end{align}
on $\Bb_\Omega(\zeta,r)$, see Theorem~\ref{thm:volume_equivalence} below. Hence Theorem~\ref{thm:main} is a true generalization of Li's theorem. 
\end{remark}

In the continuous category, Theorem~\ref{thm:main} simplifies to the following.

\begin{theorem}\label{thm:main_C0} Suppose $\Omega \subset \Cb^d$ is a bounded domain with bounded intrinsic geometry, $\partial \Omega$ is $\Cc^0$, and $\phi \in \Cc(\overline{\Omega})$. Then the following are equivalent: 
\begin{enumerate}
\item $H_\phi$ is a compact operator on $A^2(\Omega)$,  
\item $\phi$ is holomorphic on every analytic variety in $\partial\Omega$. 
\end{enumerate}
\end{theorem}

\begin{remark} To be precise, we say:
 \begin{enumerate}
\item $\partial \Omega$ is $\Cc^0$,  if for every point $x \in \partial\Omega$ there exists a neighborhood $U$ of $x$ and there exists a linear change of coordinates which makes $U \cap \partial \Omega$ the graph of a $\Cc^{0}$ function. 
\item $\phi$ is holomorphic on every analytic variety in $\partial\Omega$, if for every holomorphic map $F: \Db \rightarrow \partial \Omega$ the composition $\phi \circ F$ is holomorphic.  
\end{enumerate}
\end{remark} 

\begin{remark}
Theorem~\ref{thm:main_C0} is related to a number of prior results for convex domains:
\begin{enumerate} 
\item For smoothly bounded convex domains with symbols in $\Cc^\infty(\overline{\Omega})$, \v{C}u\v{c}kovi\'{c}-\c{S}ahuto\u{g}lu~\cite{CS2009} proved that (1) $\Rightarrow$ (2). 
\item For bounded convex domains with symbols in $\Cc(\overline{\Omega})$, \c{C}elik-\c{S}ahuto\u{g}lu-Straube~\cite{CSS2020, CSS2020b} proved that (1) $\Rightarrow$ (2) and also established an analogous result for Hankel operators on $(0,q)$-forms.
\item For bounded convex domains with symbols in $\Cc^1(\overline{\Omega})$, \c{C}elik-\c{S}ahuto\u{g}lu-Straube~\cite{CSS2020c}  proved that (2) $\Rightarrow$ (1) and also established an analogous result for Hankel operators on $(0,q)$-forms.

\end{enumerate}
It appears that even in the special case of convex domains the implication (2) $\Rightarrow$ (1) was unknown for symbols in $\Cc(\overline{\Omega})$. 
\end{remark}

We can also characterize the Hankel operators that extend to bounded operators. 

\begin{theorem}\label{thm:main_2} Suppose $\Omega \subset \Cb^d$ is a bounded domain with bounded intrinsic geometry and $\phi \in \Sc(\Omega)$. Then the following are equivalent: 
\begin{enumerate}
\item $H_\phi$ extends to a bounded operator on $A^2(\Omega)$,  
\item for some $r > 0$ 
$$\sup_{\zeta \in \Omega} \, \inf\left\{ \int_{\Bb_\Omega(\zeta,r)} \abs{\phi-h}^2  dV_\Omega  : h \in {\rm Hol}\left(\Bb_\Omega(\zeta,r)\right)\right\} <+\infty.$$ 
\end{enumerate}
\end{theorem}

In the last part of the paper we use the Bergman kernel $\Bf_\Omega$ to characterize the domains  $\Omega \subset \Cb^d$ with bounded intrinsic geometry where the estimates in Equation~\eqref{eqn:volume_equivalence} hold.

\begin{theorem}\label{thm:volume_equivalence}(see Theorem~\ref{thm:char_self_bd_gradient} below) Suppose $\Omega \subset \Cb^d$ is a bounded domain with bounded intrinsic geometry. Then the following are equivalent: 
\begin{enumerate}
\item $\log \Bf_\Omega(z,z)$ has self bounded gradient,
\item for every $r > 0$ there exists $C=C(r) > 1$ such that: if $\zeta \in \Omega$, then
\begin{align*}
\frac{1}{C} dV_\Omega \leq \frac{1}{\mu(\Bb_\Omega(\zeta,r))} d\mu \leq C dV_\Omega
\end{align*}
on $\Bb_\Omega(\zeta, r)$. 
\end{enumerate}
\end{theorem} 

\begin{remark} In Theorem~\ref{thm:char_self_bd_gradient} we also provide several other equivalent statements. \end{remark}

Using Theorems~\ref{thm:main} and~\ref{thm:volume_equivalence} we obtain the following direct extension of Li's theorem.  

\begin{corollary}\label{cor:equivalence_of_averages} Suppose $\Omega \subset \Cb^d$ is a bounded domain with bounded intrinsic geometry, $\log \Bf_\Omega(z,z)$ has self bounded gradient, and $\phi \in \Sc(\Omega)$. Then the following are equivalent: 
\begin{enumerate}
\item $H_\phi$ extends to a compact operator on $A^2(\Omega)$,  
\item for some $r > 0$ 
$$\lim_{\zeta \rightarrow \partial \Omega} \inf\left\{\frac{1}{\mu(\Bb_\Omega(\zeta,r))} \int_{\Bb_\Omega(\zeta,r)} \abs{\phi-h}^2  d\mu : h \in {\rm Hol}\left(\Bb_\Omega(\zeta,r)\right)\right\} = 0.$$ 
\end{enumerate}
\end{corollary}

It is known that $\log \Bf_\Omega(z,z)$ has self bounded gradient when $\Omega$ is a strongly pseudoconvex domain~\cite[Proposition 3.4]{D1994}, a pseudoconvex finite type domain in $\Cb^2$~\cite{D1997}, a Kobayashic hyperbolic convex domain~\cite[Proposition 4.6]{Z2021}, or a Kobayashi hyperbolic $\Cb$-convex domain~\cite[Proposition 4.12]{Z2021}. To the best of our knowledge, Corollary~\ref{cor:equivalence_of_averages} is new in all but the first case (which is Li's theorem). 

There also exist domains with bounded intrinsic geometry where  $\log \Bf_\Omega(z,z)$ does not have self bounded gradient, see~\cite[Proposition 1.11]{Z2021}.

Theorem~\ref{thm:volume_equivalence} says that, in general, the local $L^2$-spaces considered in Li's theorem and in Theorem~\ref{thm:main} are not uniformly comparable, but it is unclear if it is possible to construct a symbol which satisfies one condition but not the other.

\subsection{Structure of the paper} In Section~\ref{sec:prelim} we set our notations and recall a few classical results. In Section~\ref{sec:DBIG} we recall the definition of domains with bounded intrinsic geometry and some results from~\cite{Z2021}. Then we prove a number of new results about these domains. 

Sections \ref{sec:smooth_symbols}, \ref{sec:multiplication_operators}, \ref{sec:pf_of_thm_main}, and \ref{sec:pf_of_thm_main2} are devoted to the proofs of Theorems~\ref{thm:main} and~\ref{thm:main_2}. The proofs are similar to the arguments in~\cite{Li1994} (which in turn are similar to the arguments in~\cite{BBC1990, Li1992,LiLuecking1994,Luecking1992}) and in some sense the most important results of this paper are in Section~\ref{sec:DBIG} which give us the tools necessary to adapt Li's proof. 

In Section \ref{sec:multiplication_operators} we characterize the $L^2$-symbols whose multiplication operator is compact (respectively bounded). Using this result, in Section \ref{sec:smooth_symbols} we prove a sufficient condition for a $\Cc^1$-smooth symbol to have compact (respectively bounded) Hankel operator. 

The implication (1) $\Rightarrow$ (2) in Theorems~\ref{thm:main} and~\ref{thm:main_2} is fairly straightforward (using the results in Section~\ref{sec:DBIG}). To show that (2) $\Rightarrow$ (1), we construct a special decomposition of our symbol $\phi = \phi_1 + \phi_2$. Then using the results of Sections~\ref{sec:multiplication_operators} and~\ref{sec:smooth_symbols}  we show that $H_{\phi_1}$ and $H_{\phi_2}$ are both compact. 

In Section~\ref{sec:pf_of_thm_C0} we prove Theorem~\ref{thm:main_C0} using Theorem~\ref{thm:main}. Finally, in Section~\ref{sec:char_SBG} we consider domains with bounded intrinsic geometry where the standard potential of the Bergman metric has self bounded gradient. 

\subsection*{Acknowledgements} This research was partially supported by grants DMS-2105580 and DMS-2104381 from the
National Science Foundation.

\section{Preliminaries}\label{sec:prelim}

\subsection{Notations}  In this section we fix any possibly ambiguous notation. 

\subsubsection{The Bergman kernel, metric, volume, and distance:} We will use the following notations. 

\begin{definition} Suppose $\Omega \subset \Cb^d$ is a pseudoconvex domain. 
\begin{enumerate}
\item Let $\Bf_\Omega$ denote the Bergman kernel on $\Omega$,
\item let $g_\Omega$ denote the Bergman metric on $\Omega$, 
\item let $V_\Omega$ denote the volume form induced by the Bergman metric, that is 
\begin{align*}
dV_\Omega = \abs{\det\left[ g_\Omega\left( \frac{\partial}{\partial z_j}, \frac{\partial}{\partial \bar{z}_k} \right)\right]} d\mu,
\end{align*}
\item let $\dist_\Omega$ denote the distance induced by the Bergman metric, and 
\item for $\zeta \in \Omega$ and $r \geq 0$ let 
\begin{align*}
\Bb_\Omega(\zeta, r) :=\{ z \in \Omega : \dist_\Omega(z,\zeta) < r\}
\end{align*}
denote the open ball of radius $r$ centered at $\zeta$ in the Bergman distance. 
\end{enumerate}
\end{definition} 

\subsubsection{Approximate inequalities:} Given functions $f,h : X \rightarrow \Rb$ we write $f \lesssim h$ or equivalently $h \gtrsim f$ if there exists a constant $C > 0$ such that $f(x) \leq C h(x)$ for all $x \in X$. Often times the set $X$ will be or include  a set of parameters (e.g. $m \in \Nb$). 

If $f \lesssim g$ and $g \lesssim f$ we write $f \asymp g$. 

\subsubsection{The Levi form:} Given a domain $\Omega \subset \Cb^d$ and a $\Cc^2$-smooth real valued function $f : \Omega \rightarrow \Rb$, the \emph{Levi form of $f$} is 
\begin{align*}
\Levi(f) = \sum_{1 \leq j,k \leq d} \frac{\partial^2 f}{\partial z_j \partial \bar{z}_k}dz_j \otimes d\bar{z}_k. 
\end{align*}
Notice that $f$ is plurisubharmonic if $\Lc(f) \geq 0$ and, by definition,
\begin{align*}
\Levi\left(\log \Bf_\Omega(z,z)\right)=g_\Omega.
\end{align*}

\subsubsection{Norms on 1-forms and functions with self bounded gradient}\label{sec:defn_SBG} Given a 1-form $\alpha$ on a domain $\Omega \subset \Cb^d$ and a Hermitian pseudo-metric $h$ on $\Omega$, one can define the pointwise norm 
\begin{align*}
\norm{\alpha_z}_h = \sup\left\{ \abs{\alpha_z(X)} : X \in \Cb^d, \, h_z(X,X) \leq 1\right\}.
\end{align*}
Then a $\Cc^2$ plurisubharmonic function $\lambda : \Omega \rightarrow \Rb$ is said to have \emph{self bounded gradient} if 
\begin{align*}
\norm{\partial \lambda}_{\Levi(\lambda)}
\end{align*}
is uniformly bounded on $\Omega$. This is equivalent to the existence of some $C > 0$ such that
\begin{align*}
\abs{\sum_{j=1}^d \frac{\partial \lambda}{\partial z_j} X_j }^2 \leq C \sum_{j,k=1}^d  \frac{\partial^2 \lambda}{\partial z_j \partial \bar{z}_k}X_j \bar{X}_k
\end{align*}
for all $X \in \Cb^d$. 

\subsection{Solutions to the $\bar{\partial}$-equation} We will use the following existence theorem for solutions to the $\bar{\partial}$-equation.

\begin{theorem}\label{thm:existence-gen} Suppose $\Omega \subset \Cb^d$ is a bounded pseudoconvex domain, $\lambda_1 : \Omega \rightarrow \Rb$ has self bounded gradient, and $\lambda_2: \Omega \rightarrow \{-\infty\}\cup \Rb$ is plurisubharmonic. There exists $C > 0$ which only depends on 
\begin{align*}
\sup_{z \in \Omega} \norm{\partial \lambda_1}_{\Levi(\lambda_1)}
\end{align*}
such that: if $\alpha \in L^{2,{\rm loc}}_{(0,1)}(\Omega)$ and $\bar{\partial}\alpha= 0$, then there is some $u \in L^{2,{\rm loc}}(\Omega)$ with $\bar{\partial}u = \alpha$ and 
\begin{align*}
\int_\Omega \abs{u}^2 e^{-\lambda_2} d\mu \leq C\int_\Omega \norm{\alpha}_{\Levi(\lambda_1)}^2 e^{-\lambda_2} d\mu
\end{align*}
assuming the right hand side is finite. 
\end{theorem}

A proof of Theorem~\ref{thm:existence-gen} can be found in~\cite[Theorem 4.5 and Section 4.6]{MV2015}. A special case was established earlier in~\cite[Proposition 3.3]{M2001} with essentially the same argument.

\subsection{Separated sets in Riemannian manifolds} Recall that a set of points $A$ in a metric space $(X,\dist_X)$ is called \emph{$r$-separated} if $\dist_X(x_1,x_2) \geq r$ for all distinct $x_1,x_2 \in A$. We will frequently use the following observation about separated sets in Riemannian manifolds satisfying a type of bounded geometry condition. 

\begin{proposition}\label{prop:separated_sets} Suppose $(X,g)$ is a complete Riemannian manifold with bounded sectional curvature and positive injectivity radius. Let $\dist_g$ denote the distance induced by $g$ and let $\Bb_g(x,r)$ denote the open metric  ball of radius $r$ centered at $x \in X$. For any $r, R > 0$ there exists $L=L(r, R) > 0$ such that: if $A$ is a $r$-separated set in $(X,\dist_g)$, then 
\begin{align*}
\#( A \cap \Bb_g(x,R)) \leq L
\end{align*}
for any $x \in X$. 
\end{proposition}  

\begin{proof} Let $V_g$ denote the volume induced by $g$. By the Bishop-Gromov volume comparison theorem, there exists $C_1 > 0$ such that $V_g( \Bb_g(x,R+r)) \leq C_1$ for all $x \in X$. Since the injectivity radius is positive, by~\cite[Proposition 14]{Croke1980} there exists $C_2 > 0$ such that $V_g( \Bb_g(x,r/2)) \geq C_2$ for all $x \in X$. 

Fix $x \in X$ and suppose that $x_1,\dots, x_m$ are distinct points in $A \cap \Bb_g(x,R)$. Then the sets $\Bb_g(x_1,r/2), \dots, \Bb_g(x_m, r/2)$ are disjoint subsets of $\Bb_g(x,R+r)$ and so 
\begin{align*}
C_1 \geq V_g( \Bb_g(x,R+r)) \geq \sum_j V_g( \Bb_g(x_j,r/2)) \geq m C_2. 
\end{align*} 
Hence 
\begin{align*}
\# (A \cap \Bb_g(x,R))\leq \frac{C_1}{C_2}
\end{align*}
and the proof is complete. 

\end{proof}

\section{Domains with bounded intrinsic geometry}\label{sec:DBIG}

 In this section we recall the definition of domains with bounded intrinsic geometry and some results from~\cite{Z2021}. Then we prove some new results.
 
\begin{definition}\label{defn:BBG}\cite[Definition 1.1]{Z2021} A domain $\Omega \subset \Cb^d$ has \emph{bounded intrinsic geometry} if there exists a complete K\"ahler metric $g$ on $\Omega$ such that 
\begin{enumerate}[label=(b.\arabic*)]
\item\label{item:bd_sec} the metric $g$ has bounded sectional curvature and positive injectivity radius,  
\item\label{item:SBG} there exists a $\Cc^2$ function $\lambda : \Omega \rightarrow \Rb$ such that  the Levi form of $\lambda$ is uniformly bi-Lipschitz to  $g$ and $\norm{\partial \lambda}_{g}$ is bounded on $\Omega$. 
\end{enumerate}
\end{definition}

The K\"ahler metric in Definition~\ref{defn:BBG} does not have to be one of the standard invariant K\"ahler metrics, but in~\cite{Z2021} we proved that once there is some K\"ahler metric satisfying the definition, then the Bergman metric also satisfies the definition. 

 \begin{theorem}\label{thm:bergman_good}\cite[Theorem 1.2]{Z2021} If $\Omega \subset \Cb^d$ is a domain with bounded instrinsic geometry, then the Bergman metric $g_\Omega$ on $\Omega$ satisfies Definition~\ref{defn:BBG}. In particular, $\Omega$ is pseudoconvex. \end{theorem} 

We will also use the following theorem from~\cite{Z2021}.

 \begin{theorem}\label{thm:kob_vs_bergman}\cite[Theorem 1.8]{Z2021} If $\Omega \subset \Cb^d$ is a domain with bounded instrinsic geometry, then the Bergman metric and the Kobayashi metric are bi-Lipschitz equivalent. \end{theorem}

 \subsection{Solving the  $\bar{\partial}$-equation}
 
 As a corollary to Theorems~\ref{thm:bergman_good} and~\ref{thm:existence-gen} we have the following existence theorem for solutions to the $\bar{\partial}$-equation on domains with bounded intrinsic geometry. 
 
 \begin{corollary}\label{cor:existence_big-domains} Suppose $\Omega \subset \Cb^d$ is a bounded domain with bounded intrinsic geometry. Then there exists $C > 0$ such that: if $\lambda_2: \Omega \rightarrow \{-\infty\}\cup \Rb$ is plurisubharmonic and $\alpha \in L^{2,{\rm loc}}_{(0,1)}(\Omega)$ with $\bar{\partial}\alpha= 0$, then there is some $u \in L^{2,{\rm loc}}(\Omega)$ with $\bar{\partial}u = \alpha$ and 
\begin{align*}
\int_\Omega \abs{u}^2 e^{-\lambda_2} d\mu \leq C\int_\Omega \norm{\alpha}_{g_\Omega}^2 e^{-\lambda_2} d\mu
\end{align*}
assuming the right hand side is finite. 
\end{corollary}

\begin{proof} Let $\lambda$ be a $\Cc^2$ function satisfying Definition~\ref{defn:BBG} for the Bergman metric. Then apply Theorem~\ref{thm:existence-gen} with $\lambda_1 = \lambda$.
\end{proof}

\subsection{Discretization} As a corollary to Theorem~\ref{thm:bergman_good} and Proposition~\ref{prop:separated_sets} we have the following useful discretization of a domain with bounded intrinsic geometry. 
 
 \begin{corollary}\label{cor:discretization} Suppose $\Omega \subset \Cb^d$ is a bounded domain with bounded intrinsic geometry. For any $r > 0$ there exists a sequence of distinct points $(\zeta_m)_{m \geq 1}$ in $\Omega$ with the following properties: 
 \begin{enumerate}
 \item $\{\zeta_m : m \geq 1\}$ is $r$-separated in $(\Omega, \dist_\Omega)$, 
 \item $\Omega = \cup_m \Bb_\Omega(\zeta_m, r)$, and
 \item for any $R > 0$,  $\sup_{z \in \Omega} \#\{ m : \zeta_m \in \Bb_\Omega(z, R) \} <+\infty$.
 \end{enumerate}
 \end{corollary}
 
 \begin{proof} Let $A \subset \Omega$ be a maximal $r$-separated set in $\Omega$ (which exists by Zorn's lemma). Since the Bergman metric is complete, the metric space $(\Omega, \dist_\Omega)$ is unbounded (by the Hopf-Rinow theorem). So the set $A$ must be infinite. By Proposition~\ref{prop:separated_sets}, the set $A$ is countable. Hence $A = \{\zeta_m : m \geq 1\}$ for some sequence $(\zeta_m)_{m \geq 1}$ of distinct points. 
 
Then part (1) follows from the definition of $A$. Part (2) follows from the maximality of $A$: if there exists $w \in \Omega \setminus \cup_m \Bb_\Omega(\zeta_m, r)$, then $A \cup \{w\}$ would also be $r$-separated. Part (3) follows  from Proposition~\ref{prop:separated_sets}. 
 
 \end{proof}

\subsection{Estimates on the Bergman kernel and metric}\label{sec:estimates} In this subsection we recall two results from~\cite{Z2021} and then use them to derive a number of new estimates for the Bergman kernel and metric. 

Suppose, for the rest of this subsection, that $\Omega \subset \Cb^d$ is a bounded domain with bounded intrinsic geometry.  

Combining Theorem~\ref{thm:bergman_good} with deep results of Wu-Yau~\cite{WY2020} and Shi~\cite{Shi1997}, yields the following. 

\begin{theorem}\label{thm:good_charts}\cite[Theorem 5.1, Theorem 10.1]{Z2021} There exists $C_1 > 1$ such that: for every $\zeta \in \Omega$ there is a holomorphic embedding $\Phi_\zeta : \Bb \rightarrow \Omega$ with $\Phi_\zeta(0) = \zeta$,
\begin{align*}
\frac{1}{C_1^2} g_{\Euc} \leq \Phi_\zeta^* g_\Omega \leq C_1^2 g_{\Euc} 
\end{align*}
on $\Bb$, and 
\begin{align}
\label{eqn:dist_comp}
\frac{1}{C_1}\norm{w_1-w_2}_{2} \leq \dist_\Omega\left(\Phi_\zeta(w_1), \Phi_\zeta(w_2)\right) \leq C_1 \norm{w_1-w_2}_{2}
\end{align}
for all $w_1,w_2 \in \Bb$.
\end{theorem} 

\begin{remark} Notice that Equation~\eqref{eqn:dist_comp} implies that 
\begin{align*}
\Bb_\Omega\left(\zeta, \frac{r}{C_1}\right) \subset \Phi_\zeta(r\Bb) \subset \Bb_\Omega\left(\zeta, C_1 r\right)
\end{align*}
when $r < \frac{1}{C_1}$. 
\end{remark}

Using the embeddings in Theorem~\ref{thm:good_charts}, we define
\begin{align*}
\beta_\zeta &: \Bb \times \Bb \rightarrow \Cb \\
\beta_\zeta&(u,w) = \Bf_\Omega(\Phi_\zeta(u), \Phi_\zeta(w)) \det \Phi_\zeta^\prime(u)\overline{\det \Phi_\zeta^\prime(w)}.
\end{align*}
These functions have the following uniform estimates. 

\begin{theorem}\label{thm:Bergman_kernel_est}\cite[Theorem 9.1, Theorem 10.1]{Z2021} \ 
\begin{enumerate}
\item There exists $C_2 > 1$ such that: 
\begin{align*}
C_2^{-1} \leq \beta_\zeta(w,w) \leq C_2 
\end{align*}
for all $\zeta \in \Omega$ and $w \in \Bb$. 
\item For every $\delta \in (0,1)$ and multi-indices $a,b$ there exists $C_{\delta, a,b} > 0$ such that 
\begin{align*}
\abs{\frac{\partial^{\abs{a}+\abs{b}} \beta_\zeta(u,w)}{\partial u^{a} \partial \bar{w}^b}} \leq C_{\delta, a,b}
\end{align*}
for all $\zeta \in \Omega$ and $u,w \in \delta \Bb$. 
\end{enumerate}
\end{theorem} 

Theorem~\ref{thm:Bergman_kernel_est} implies the following off-diagonal estimates near the diagonal. 

\begin{proposition}\label{prop:off_diagonal_estimate} There exists $r_0 > 0$ and $C_3 > 1$ such that: If $\zeta \in \Omega$ and $z \in \Phi_\zeta(r_0\Bb)$, then 
\begin{align*}
\frac{1}{C_3}\Bf_\Omega(z,z)\Bf_\Omega(\zeta,\zeta) \leq \abs{\Bf_\Omega(z,\zeta)}^2 \leq C_3 \Bf_\Omega(z,z)\Bf_\Omega(\zeta,\zeta).
\end{align*}
\end{proposition}

\begin{proof} By Theorem~\ref{thm:Bergman_kernel_est} part (1) and (2) there exists $r_0 > 0$ such that 
\begin{align*}
\frac{1}{2C_2^2} \leq \abs{\beta_\zeta(w,0)}^2 \leq 2C_2^2 
\end{align*}
for all $w \in r_0 \Bb$. Also, 
\begin{align*}
 \frac{1}{C_2} \abs{\det \Phi_\zeta^\prime(w)}^{-2} \leq \Bf_\Omega(\Phi_\zeta(w), \Phi_\zeta(w))  \leq C_2\abs{\det \Phi_\zeta^\prime(w)}^{-2}
\end{align*}
for all $w \in \Bb$. Now if $z = \Phi_\zeta(w)$ and $w \in r_0 \Bb$, then 
\begin{align*}
\abs{\Bf_\Omega(z,\zeta)}^2 = \abs{\beta_\zeta(w,0)}^2  \abs{\det \Phi_\zeta^\prime(w)}^{-2} \abs{\det \Phi_\zeta^\prime(0)}^{-2} 
\end{align*}
and so 
\begin{align*}
\frac{1}{C_3}\Bf_\Omega(z,z)\Bf_\Omega(\zeta,\zeta) \leq \abs{\Bf_\Omega(z,\zeta)}^2 \leq C_3 \Bf_\Omega(z,z)\Bf_\Omega(\zeta,\zeta)
\end{align*}
where $C_3 = 2 C_2^4$. 
\end{proof}

\begin{proposition}\label{prop:sMVT} There exists $C_4 > 0$ such that: if $r < C_1$ and $u : \Omega \rightarrow [0,\infty)$ is a function with $\log(u)$ plurisubharmomic, then 
\begin{align}
\label{eqn:sMVT}
u(\zeta) \leq \frac{C_4}{r^{2d}} \Bf_\Omega(\zeta,\zeta) \int_{\Bb_\Omega(\zeta,r)} u \, d\mu
\end{align}
for all $\zeta \in \Omega$. 
\end{proposition}

\begin{proof} Theorem~\ref{thm:good_charts} implies that $\Phi_\zeta\left( \frac{r}{C_1} \Bb\right) \subset \Bb_\Omega(\zeta, r)$ and so 
\begin{align*}
\int_{\Bb_\Omega(\zeta,r)} u \, d\mu \geq \int_{\frac{r}{C_1}\Bb} (u \circ \Phi_\zeta)(w) \abs{\det \Phi_\zeta^\prime(w)}^2 d\mu(w).
\end{align*}
Notice that $\log  (u) \circ \Phi_\zeta+ \log  \abs{\det \Phi_\zeta^\prime}^2$ is plurisubharmonic and so 
\begin{align*}
\exp\left( \log  (u) \circ \Phi_\zeta+ \log  \abs{\det \Phi_\zeta^\prime}^2\right) = (u \circ \Phi_\zeta)\cdot \abs{\det \Phi_\zeta^\prime}^2
\end{align*}
is also plurisubharmonic. So by the mean value theorem for plurisubharmonic functions
\begin{align*}
 u(\zeta) \abs{\det \Phi_\zeta^\prime(0)}^2 \leq \frac{C_1^{2d}}{r^{2d} \mu(\Bb)} \int_{\Bb_\Omega(\zeta,r)} u \, d\mu. 
\end{align*}
So by Theorem~\ref{thm:Bergman_kernel_est}, 
\begin{align*}
u(\zeta) \leq \frac{C_4}{r^{2d}} \Bf_\Omega(\zeta,\zeta) \int_{\Bb_\Omega(\zeta,r)} u \, d\mu
\end{align*}
where $C_4 := \frac{C_1^{2d}C_2}{\mu(\Bb)}$.
\end{proof} 

As a consequence of the above proposition, the Bergman kernel has the following local positivity of mass. 

\begin{proposition} If $r < C_1$ and $\zeta \in \Omega$, then 
\begin{align*}
\int_{\Omega} \abs{\Bf_\Omega(z,\zeta)}^2 d\mu(z) \leq \frac{C_4}{r^{2d}} \int_{\Bb_\Omega(\zeta,r)} \abs{\Bf_\Omega(z,\zeta)}^2 d\mu(z).
\end{align*}
\end{proposition}

\begin{proof} Apply Equation~\eqref{eqn:sMVT} to  $u= \abs{\Bf_\Omega(\cdot, \zeta)}^2$ and recall that 
\begin{align*}
\Bf_\Omega(\zeta,\zeta) =  \int_{\Omega} \abs{\Bf_\Omega(z,\zeta)}^2 d\mu(z).
\end{align*}
\end{proof}

We also obtain the following estimate on the volume form induced by the Bergman metric. 

\begin{proposition}\label{prop:volume_comp} There exists $C_5 > 1$ such that
\begin{align*}
\frac{1}{C_5}\Bf_\Omega(z,z)d\mu(z) \leq dV_\Omega(z) \leq C_5\Bf_\Omega(z,z)d\mu(z)
\end{align*}
on $\Omega$. 
\end{proposition}

\begin{proof} By Theorem~\ref{thm:good_charts}
\begin{align*}
\frac{1}{C_1^{2d}} \leq \abs{\det \left[(\Phi_\zeta^*g_\Omega)_0\left( \frac{\partial}{\partial z_j}, \frac{\partial}{\partial \bar{z}_k} \right)\right]} \leq C_1^{2d}
\end{align*}
and so
\begin{align*}
\frac{1}{C_1^{2d}} \abs{\det \Phi_\zeta^\prime(0)}^{-2} \leq \abs{\det\left[ g_{\Omega,\zeta}\left( \frac{\partial}{\partial z_j}, \frac{\partial}{\partial \bar{z}_k} \right)\right]} \leq C_1^{2d}\abs{\det \Phi_\zeta^\prime(0)}^{-2}.
\end{align*}
Hence Theorem~\ref{thm:Bergman_kernel_est} implies that 
\begin{align*}
\frac{1}{C_5}\Bf_\Omega(z,z)d\mu(z) \leq dV_\Omega(z) \leq C_5\Bf_\Omega(z,z)d\mu(z)
\end{align*}
where $C_5 = C_1^{2d}C_2$. 
\end{proof}

For each $\zeta \in \Omega$, consider the function 
\begin{align*}
s_\zeta = \frac{1}{\sqrt{\Bf_\Omega(\zeta,\zeta)}} \Bf_\Omega(\cdot, \zeta).
\end{align*}
Then $s_\zeta \in A^2(\Omega)$ and $\norm{s_\zeta}_2 = 1$. As an application of Proposition~\ref{prop:sMVT} and the completeness of the Bergman metric, we have the following convergence result. 

\begin{proposition}\label{prop:weak_convergence} If $\zeta_m \rightarrow \partial \Omega$, then $s_{\zeta_m}$ converges locally uniformly to the zero function. \end{proposition}

\begin{proof} Suppose not. Then after passing to a subsequence we can suppose that $s_{\zeta_m} \rightarrow f$ locally uniformly where $f$ is holomorphic and non-zero. By Fatou's lemma, $\int_\Omega \abs{f}^2 d\mu \leq 1$. 

Fix a sequence of compact sets $(K_m)_{m \geq 1}$ in $\Omega$ with $\int_{K_m} \abs{f}^2 d\mu \rightarrow \int_\Omega \abs{f}^2 d\mu$. Replacing $(s_{\zeta_m})_{m \geq 1}$ with a subsequence, we can assume 
\begin{align*}
\lim_{m \rightarrow \infty} \int_{K_m} \abs{s_{\zeta_{m}}-f}^2 d\mu = 0.
\end{align*}
Then 
\begin{align*}
\limsup_{m \rightarrow \infty}\norm{s_{\zeta_m}-f}_2 & = \limsup_{m \rightarrow \infty}\norm {(s_{\zeta_m}-f) \mathds{1}_{\Omega \setminus K_m}}_2 \leq \limsup_{m \rightarrow \infty}\norm {s_{\zeta_m}\mathds{1}_{\Omega \setminus K_m}}_2 + \norm{f \mathds{1}_{\Omega \setminus K_m}}_2 \\
& = \limsup_{m \rightarrow \infty} 1 - \norm{s_{\zeta_m} \mathds{1}_{K_m}}_2 = 1 -\norm{f}_2. 
\end{align*}

Fix $r < C_1$. Applying Proposition~\ref{prop:sMVT} to $\abs{f}^2$ yields 
\begin{align*}
\abs{f(\zeta_m)} \leq \frac{C_4}{r^{2d}}\sqrt{\Bf_\Omega(\zeta_m,\zeta_m)} \left(\int_{\Bb_\Omega(\zeta_m,r)} \abs{f}^2 d\mu\right)^{1/2}
\end{align*}
for all $m \geq 1$. Since the Bergman metric is proper (by the Hopf-Rinow theorem) and $\zeta_m \rightarrow \partial \Omega$, for any compact set $K \subset \Omega$ the sets 
\begin{align*}
\Bb_\Omega(\zeta_m,r) \cap K
\end{align*}
are eventually empty. Then, since $f \in L^2(\Omega)$, we have 
\begin{align*}
\lim_{m \rightarrow \infty} \int_{\Bb_\Omega(\zeta_m,r)} \abs{f}^2 d\mu = 0
\end{align*}
and so 
\begin{align*}
\lim_{m \rightarrow \infty} \frac{\abs{f(\zeta_m)}}{\sqrt{\Bf_\Omega(\zeta_m,\zeta_m)}}=0. 
\end{align*}

Finally, let 
\begin{align*}
h_m = \frac{1}{1-(1/2)\norm{f}_2} (s_{\zeta_{m}}-f).
\end{align*}
Then for $m$ large we have $\norm{h_m}_2 < 1$ and 
\begin{align*}
\abs{h_m(\zeta_{m})} > \sqrt{\Bf_\Omega(\zeta_{m},\zeta_{m})}
\end{align*}
which is impossible since
\begin{align*}
\sqrt{\Bf_\Omega(\zeta,\zeta)} = \sup \left\{ \abs{h(\zeta)} : h \in A^2(\Omega), \norm{h}_2 \leq 1\right\}. 
\end{align*}

\end{proof}

\section{Multiplication operators}\label{sec:multiplication_operators}

Given $\phi \in L^{2}(\Omega)$, the associated \emph{multiplication operator} $M_\phi$ has domain 
\begin{align*}
{\rm dom}(M_\phi) = \left\{ f \in A^2(\Omega) : \phi \cdot f \in L^2(\Omega)\right\}
\end{align*}
and is defined by 
\begin{align*}
M_\phi(f) = \phi \cdot f.
\end{align*}

\begin{proposition}\label{prop:multiplication_operators} Suppose  $\Omega \subset \Cb^d$ is a bounded domain with bounded intrinsic geometry and $\phi \in L^2(\Omega)$. Then:
\begin{enumerate}
\item The following are equivalent: 
\begin{enumerate}
\item  there exists $r > 0$ such that
\begin{align*}
\sup_{\zeta \in \Omega}  \int_{\Bb_\Omega(\zeta,r)} \abs{\phi}^2  dV_\Omega  <+\infty,
\end{align*}
\item ${\rm dom}(M_\phi) = A^2(\Omega)$ and $M_\phi : A^2(\Omega) \rightarrow L^2(\Omega)$ is bounded.
\end{enumerate}
\item The following are equivalent: 
\begin{enumerate}
\item[(a')]  there exists $r > 0$ such that
\begin{align*}
\lim_{\zeta \rightarrow \partial \Omega}  \int_{\Bb_\Omega(\zeta,r)} \abs{\phi}^2  dV_\Omega  = 0,
\end{align*}
\item[(b')] ${\rm dom}(M_\phi) = A^2(\Omega)$ and $M_\phi : A^2(\Omega) \rightarrow L^2(\Omega)$ is compact. 
\end{enumerate}
\end{enumerate}
\end{proposition} 

The rest of the section is devoted to the proof of the theorem. 

\begin{lemma} (b') $\Rightarrow$ (a'). \end{lemma} 

\begin{proof} As in Section~\ref{sec:estimates}, for each $\zeta \in \Omega$, consider the function 
\begin{align*}
s_\zeta = \frac{1}{\sqrt{\Bf_\Omega(\zeta,\zeta)}} \Bf_\Omega(\cdot, \zeta) \in A^2(\Omega). 
\end{align*}
Using Propositions~\ref{prop:off_diagonal_estimate} and~\ref{prop:volume_comp} we can fix $r > 0$ and $C > 0$ such that 
\begin{align*}
\frac{1}{C} dV_\Omega \leq \abs{s_{\zeta}}^2 d\mu \leq C dV_\Omega
\end{align*}
on $\Bb_\Omega(\zeta,r)$.

Fix a sequence $(\zeta_m)_{m \geq 1}$ where $\zeta_m \rightarrow \partial \Omega$ and  
\begin{align*}
\limsup_{\zeta \rightarrow \partial \Omega}  \int_{\Bb_\Omega(\zeta,r)} \abs{\phi}^2  dV_\Omega  = \lim_{m \rightarrow \infty}  \int_{\Bb_\Omega(\zeta_m,r)} \abs{\phi}^2  dV_\Omega.
\end{align*}
By Proposition~\ref{prop:weak_convergence} the sequence $(s_{\zeta_m})_{m \geq 1}$ converges weakly to 0. Since $M_\phi$ is compact, then 
\begin{align*}
\lim_{m \rightarrow \infty}  \int_{\Omega} \abs{\phi \cdot s_\zeta}^2  d\mu = 0.
\end{align*}
 Then
\begin{align*}
\lim_{m \rightarrow \infty} \int_{\Bb_\Omega(\zeta_m,r)} \abs{\phi}^2  dV_\Omega  \leq C \lim_{m \rightarrow \infty}  \int_{\Omega} \abs{\phi \cdot s_\zeta}^2  d\mu  =0
\end{align*}
and the proof is complete. 
\end{proof}

\begin{lemma} (b) $\Rightarrow$ (a). \end{lemma} 

\begin{proof} Very similar to the proof that (b') $\Rightarrow$ (a').  \end{proof}

\begin{lemma}\label{lem:multiplication_a_implies_b} (a) $\Rightarrow$ (b). \end{lemma} 

\begin{proof} Fix $f \in A^2(\Omega)$. By Corollary~\ref{cor:discretization} there exists a sequence $(\zeta_m)_{m \geq 1}$ of distinct points in $\Omega$ such that 
\begin{enumerate}
\item $\{\zeta_m : m \geq 1\}$ is $r$-separated with respect to the Bergman distance,
\item $\cup_m \Bb_\Omega(\zeta_m, r) = \Omega$, and 
 \item $L:=\sup_{z \in \Omega} \#\{ m : \zeta_m \in \Bb_\Omega(z, 2r) \} <+\infty$.
\end{enumerate}

Applying Proposition~\ref{prop:sMVT} to $\abs{f}^2$  yields: if $z \in \Bb_\Omega(\zeta_m, r)$, then
\begin{align*}
\abs{f(z)}^2 \lesssim \Bf_\Omega(z,z) \int_{\Bb_\Omega(z, r)} \abs{f}^2 d\mu \leq  \Bf_\Omega(z,z) \int_{\Bb_\Omega(\zeta_m, 2r)} \abs{f}^2 d\mu.
\end{align*}
So by Proposition~\ref{prop:volume_comp}
\begin{align*}
\int_{\Bb_\Omega(\zeta_m,r)} & \abs{\phi \cdot f}^2 d\mu \lesssim  \left(\int_{\Bb_\Omega(\zeta_m,r)} \abs{\phi}^2 \Bf_\Omega(z,z) d \mu\right) \left( \int_{\Bb_\Omega(\zeta_m, 2r)} \abs{f}^2 d\mu\right)  \\
& \lesssim \left( \int_{\Bb_\Omega(\zeta_m,r)} \abs{\phi}^2  dV_\Omega\right) \left(\int_{\Bb_\Omega(\zeta_m, 2r)} \abs{f}^2 d\mu\right) \\
& \lesssim \int_{\Bb_\Omega(\zeta_m, 2r)} \abs{f}^2 d\mu.
\end{align*}
Hence 
\begin{align*}
\int_\Omega \abs{\phi \cdot f}^2 d\mu \lesssim \sum_{m} \int_{\Bb_\Omega(\zeta_m,r)} \abs{\phi \cdot f}^2 d\mu \lesssim \sum_{m} \int_{\Bb_\Omega(\zeta_m,2r)} \abs{f}^2 d\mu \leq L \int_\Omega \abs{f}^2 d\mu.
\end{align*}

Since $f \in A^2(\Omega)$ was arbitrary, ${\rm dom}(M_\phi) = A^2(\Omega)$ and  $M_\phi : A^2(\Omega) \rightarrow L^2(\Omega)$ is bounded.

\end{proof} 

\begin{lemma} (a') $\Rightarrow$ (b'). \end{lemma} 

\begin{proof} It is enough to fix a sequence $(f_n)_{n \geq 1}$ of unit vectors in $A^2(\Omega)$ which converges weakly to 0 and show that $M_\phi(f)$ converges strongly to $0$.  

As in the proof of Lemma~\ref{lem:multiplication_a_implies_b}, there exists a sequence $(\zeta_m)_{m \geq 1}$ of distinct points in $\Omega$ such that 
\begin{enumerate}
\item $\{\zeta_m : m \geq 1\}$ is $r$-separated with respect to the Bergman distance,
\item $\cup_m \Bb_\Omega(\zeta_m, r) = \Omega$, and 
 \item $L:=\sup_{z \in \Omega} \#\{ m : \zeta_m \in \Bb_\Omega(z, 2r) \} <+\infty$.
\end{enumerate}
Further, arguing as in Lemma~\ref{lem:multiplication_a_implies_b} we have
\begin{align*}
\int_{\Bb_\Omega(\zeta_m,r)} & \abs{\phi \cdot f_n}^2 d\mu \lesssim \left( \int_{\Bb_\Omega(\zeta_m,r)} \abs{\phi}^2  dV_\Omega\right) \left(\int_{\Bb_\Omega(\zeta_m, 2r)} \abs{f_n}^2 d\mu\right).
\end{align*}

Fix $\epsilon > 0$. Since 
\begin{align*}
\lim_{\zeta \rightarrow \partial \Omega}  \int_{\Bb_\Omega(\zeta,r)} \abs{\phi}^2  dV_\Omega  = 0,
\end{align*}
there exists $M > 0$ such that 
\begin{align*}
\int_{\Bb_\Omega(\zeta_m,r)} \abs{\phi}^2  dV_\Omega <  \epsilon
\end{align*}
for all $m > M$. Since $f_n \in A^2(\Omega)$ converges to 0 weakly, $f_n$ converges locally uniformly to 0. Hence
\begin{align*}
\lim_{n \rightarrow \infty} \sum_{m \leq M} \int_{ \Bb_\Omega(\zeta_m,2r)} \abs{f_n}^2 d\mu =0.
\end{align*}
Then
\begin{align*}
\limsup_{n \rightarrow \infty} & \int_\Omega \abs{\phi \cdot f_n}^2 d\mu \leq \limsup_{n \rightarrow \infty}\sum_{m} \int_{\Bb_\Omega(\zeta_m,r)}  \abs{\phi \cdot f_n}^2 d\mu \\
& \lesssim \limsup_{n \rightarrow \infty}\sum_{m}\left( \int_{\Bb_\Omega(\zeta_m,r)} \abs{\phi}^2  dV_\Omega\right) \left(\int_{\Bb_\Omega(\zeta_m, 2r)} \abs{f_n}^2 d\mu\right) \\
& = \limsup_{n \rightarrow \infty}\sum_{m>M}\left( \int_{\Bb_\Omega(\zeta_m,r)} \abs{\phi}^2  dV_\Omega\right) \left(\int_{\Bb_\Omega(\zeta_m, 2r)} \abs{f_n}^2 d\mu\right) \\
& \leq \limsup_{n \rightarrow \infty}\sum_{m>M} \epsilon \left(\int_{\Bb_\Omega(\zeta_m, 2r)} \abs{f_n}^2 d\mu\right) \leq  \limsup_{n \rightarrow\infty}\, \epsilon L \int_\Omega \abs{f_n}^2 d \mu = \epsilon L. 
\end{align*}
Since $\epsilon > 0$ was arbitrary, $M_\phi(f_n)=\phi \cdot f_n$ converges strongly to 0. 
\end{proof}

\section{Smooth symbols}\label{sec:smooth_symbols}

Using Proposition~\ref{prop:multiplication_operators} we establish a sufficient condition of a $\Cc^1$-smooth symbol to have compact (respectively bounded) Hankel operator. 

\begin{proposition}\label{prop:smooth_symbols} Suppose  $\Omega \subset \Cb^d$ is a bounded domain with bounded intrinsic geometry and $\phi \in \Cc^1(\Omega) \cap \Sc(\Omega)$. 
\begin{enumerate}
\item If there exists $r > 0$ such that
\begin{align*}
\sup_{\zeta \in \Omega} \int_{\Bb_\Omega(\zeta,r)} \norm{\bar{\partial} \phi}_{g_\Omega}^2 dV_\Omega <+\infty,
\end{align*}
then $H_\phi$ extends to a bounded operator on $A^2(\Omega)$. 
\item If there exists $r > 0$ such that
\begin{align*}
\lim_{\zeta \rightarrow \partial \Omega} \int_{\Bb_\Omega(\zeta,r)}  \norm{\bar{\partial} \phi}_{g_\Omega}^2  dV_\Omega= 0,
\end{align*}
then $H_\phi$ extends to a compact operator on $A^2(\Omega)$. 
\end{enumerate}
\end{proposition} 

\begin{proof} Let $M : A^2(\Omega) \rightarrow L^2(\Omega)$ be the multiplication operator 
\begin{align*}
M(f) =  \norm{\bar{\partial} \phi}_{g_\Omega} \cdot f.
\end{align*}

 Fix $f \in {\rm dom}(H_\phi)$. By definition 
\begin{align*}
\norm{H_\phi(f)}_2 = \min_{h \in A^2(\Omega)} \norm{f\phi-h}_2.
\end{align*}
Further, by Corollary~\ref{cor:existence_big-domains} there exists $C > 0$, independent of $f$, and some $u \in L^{2}(\Omega)$ with $\bar{\partial} u = f\bar{\partial} \phi$ and 
\begin{align*}
\int_\Omega \abs{u}^2 d\mu \leq  C \int_\Omega \abs{f}^2\norm{ \bar{\partial} \phi}^2_{g_\Omega} d\mu  =   C\norm{M(f)}_2^2
\end{align*}
Then $h : = f \phi - u \in A^2(\Omega)$ and so 
\begin{align*}
\norm{H_\phi(f)}_2  \leq  \norm{f \phi-h}_2 =\norm{u}_2 \leq \sqrt{C} \norm{M(f)}_2. 
\end{align*}

So Proposition~\ref{prop:multiplication_operators} immediately implies the result.

\end{proof}

\section{Proof of Theorem~\ref{thm:main}}\label{sec:pf_of_thm_main}

We are now ready to prove Theorem~\ref{thm:main} which we restate here. 

\begin{theorem}Suppose $\Omega \subset \Cb^d$ is a bounded domain with bounded intrinsic geometry and $\phi \in \Sc(\Omega)$. Then the following are equivalent: 
\begin{enumerate}
\item $H_\phi$ extends to a compact operator 
\item for some $r > 0$ 
$$\lim_{\zeta \rightarrow \partial \Omega} \inf\left\{ \int_{\Bb_\Omega(\zeta,r)} \abs{\phi-h}^2  dV_\Omega  : h \in {\rm Hol}\left(\Bb_\Omega(\zeta,r)\right)\right\} = 0.$$ 
\end{enumerate}
\end{theorem}

For the rest of the section suppose that  $\Omega \subset \Cb^d$ is a bounded domain with bounded intrinsic geometry and $\phi \in \Sc(\Omega)$. 

\subsection{(1) implies (2)} Suppose that $H_\phi$ extends to a compact operator $\hat{H}_\phi$ on $A^2(\Omega)$. Fix a sequence $(\zeta_m)_{m \geq 1}$ converging to $\partial \Omega$. By Proposition~\ref{prop:weak_convergence} the functions $s_{\zeta_m} \in L^2(\Omega)$ converge weakly to zero. Hence 
\begin{align*}
\lim_{m \rightarrow \infty} \norm{\hat{H}_\phi(s_{\zeta_m})}_2 = 0.
\end{align*} 

By Proposition~\ref{prop:off_diagonal_estimate} we can fix $r > 0$ and $C > 1$ such that 
\begin{align}
\label{eqn:off-diag-in-pf-thm-main}
\frac{1}{C}\abs{\Bf_\Omega(z,\zeta)}^2 \leq \Bf_\Omega(z,z)\Bf_\Omega(\zeta,\zeta) \leq C \abs{\Bf_\Omega(z,\zeta)}^2
\end{align}
for all $\zeta \in \Omega$ and $z \in \Bb_\Omega(\zeta,r)$. By increasing $C >1$ and using Propositions~\ref{prop:volume_comp} and~\ref{prop:off_diagonal_estimate} we may also assume that 
\begin{align}
\label{eqn:vol-est-in-pf-thm-main}
\frac{1}{C} dV_\Omega \leq \abs{s_{\zeta}(z)}^2 d\mu \leq C dV_\Omega
\end{align}
on each $\Bb_\Omega(\zeta,r)$. Notice that this implies that each $s_{\zeta_m}$ is non-vanishing on $\overline{\Bb_\Omega(\zeta_m,r)}$. 

By assumption, ${\rm dom}(H_\phi)$ is dense in $A^2(\Omega)$. So for each $m$ we can find a sequence $(f_{m,k})_{k \geq 1}$ in  ${\rm dom}(H_\phi)$ converging to $s_{\zeta_m}$ in $A^2(\Omega)$. Then $(f_{m,k})_{k \geq 1}$ converges uniformly to $s_{\zeta_m}$ on $\overline{\Bb_\Omega(\zeta_m,r)}$. Since  $s_{\zeta_m}$ is non-vanishing on $\overline{\Bb_\Omega(\zeta_m,r)}$, we can then pick $k_m$ such that 
\begin{align*}
\frac{1}{C} \abs{f_{m,k_m}} \leq \abs{s_{\zeta_m}} \leq C \abs{f_{m,k_m}}
\end{align*}
on $\Bb_\Omega(\zeta_m, r)$. By possibly increasing $k_m$ further we may also assume that 
\begin{align}
\label{eqn:norm-to-zero-in-pf-thm-main}
\lim_{m \rightarrow \infty} \norm{H_\phi(f_{m,k_m})}_2 = 0.
\end{align}

Let $f_m: = f_{m,k_m}$. Since $f_m$ is non-vanishing on $\Bb_\Omega(\zeta_m, r)$, the function
\begin{align*}
h_m := f_m^{-1} P_\Omega(\phi f_m)
\end{align*}
is in $ {\rm Hol} \left( \Bb_\Omega(\zeta_m,r) \right)$. Then by Equations~\eqref{eqn:vol-est-in-pf-thm-main} and~\eqref{eqn:norm-to-zero-in-pf-thm-main}
\begin{align*}
\lim_{m \rightarrow \infty} & \int_{\Bb_\Omega(\zeta_m, r)} \abs{ \phi - h_m }^2  dV_\Omega = \lim_{m \rightarrow \infty} \int_{\Bb_\Omega(\zeta_n, r)} \abs{ \phi f_m - P_\Omega(\phi f_m)}^2\abs{f_m}^{-2} dV_\Omega  \\
& \leq \lim_{m \rightarrow \infty}C^2 \int_{\Bb_\Omega(\zeta_m, r)} \abs{ H_\phi(f_m)}^2 d\mu \leq \lim_{m \rightarrow \infty} C^2\norm{H_\phi(f_m)}^2_2 = 0. 
\end{align*}

Since $(\zeta_m)_{m \geq 1}$ was an arbitrary sequence converging to $\partial \Omega$, this completes the proof of this direction.

\subsection{(2) implies (1)} Suppose that there exists $r > 0$ such that
$$\lim_{\zeta \rightarrow \partial \Omega} \inf\left\{ \int_{\Bb_\Omega(\zeta,r)} \abs{\phi-h}^2  dV_\Omega  : h \in {\rm Hol}\left(\Bb_\Omega(\zeta,r)\right)\right\} = 0.$$ 
Without loss of generality we can assume $r < 1$. 

Let $C_1 > 1$ and $\{ \Phi_\zeta : \zeta \in \Omega\}$ satisfy Theorem~\ref{thm:good_charts}. Then fix $r_1 < \frac{r}{C_1^2}$. By Corollary~\ref{cor:discretization} there exists a sequence $(\zeta_m)_{m \geq 1}$ of distinct points in $\Omega$ such that 
\begin{enumerate}
\item $\{\zeta_m : m \geq 1\}$ is $r_1$-separated with respect to the Bergman distance,
\item $\cup_m \Bb_\Omega(\zeta_m, r_1) = \Omega$, and 
 \item $L:=\sup_{z \in \Omega} \#\{ m : \zeta_m \in \Bb_\Omega(z, 3r/2) \} <+\infty$.
\end{enumerate}

Since the Bergman metric is a complete Riemannian metric and hence proper, we must have $\zeta_m \rightarrow \partial \Omega$. Then for each $m \geq 1$, there is some $h_m \in  {\rm Hol} \left( \Bb_\Omega(\zeta_m,r)\right)$ such that if 
\begin{align*}
\epsilon_m: =\left(\int_{\Bb_\Omega(\zeta_m, r)} \abs{ \phi - h_m }^2  dV_\Omega\right)^{1/2}, 
\end{align*}
then $\lim_{m \rightarrow \infty} \epsilon_m = 0$. 

Then fix a compactly supported smooth function $\chi : \Bb \rightarrow [0,1]$ such that $\chi \equiv 1$ on $C_1r_1 \Bb$ and ${\rm supp}(\chi) \subset \frac{r}{C_1} \Bb$. Then define $\chi_m := \chi \circ \Phi_{\zeta_m}^{-1}$. Notice that Theorem~\ref{thm:good_charts} implies that
\begin{align*}
\Bb_\Omega(\zeta_m, r_1) \subset \Phi_{\zeta_m}\left( C_1 r_1 \Bb\right) \subset  \chi_m^{-1}(1)  
\end{align*}
and 
\begin{align*}
 {\rm supp}(\chi_m) \subset \Phi_{\zeta_m}\left(\frac{r}{C_1} \Bb\right) \subset  \Bb_\Omega(\zeta_m, r).
\end{align*}
Further, by Theorem~\ref{thm:good_charts}
\begin{align*}
\norm{ \bar{\partial} \chi_m}_{g_\Omega}=\norm{ \bar{\partial} \chi}_{\Phi_\zeta^*g_\Omega} \leq C_1 \norm{ \bar{\partial} \chi}_{2} \lesssim 1. 
\end{align*}
Next, let $\hat{\chi}_m :=  \frac{1}{\sum_n \chi_n}\chi_m$. Then
\begin{align}
\label{eq:gradient_estimate}
\norm{ \bar{\partial} \hat{\chi}_m}_{g_\Omega} = \norm{ \frac{1}{\sum_n \chi_n}\bar{\partial}\chi_m - \frac{\chi_m}{(\sum_n \chi_n)^2} \sum_n \bar{\partial} \chi_n}_{g_\Omega}  \leq (L+1) \sup_{n \geq 1} \norm{\bar{\partial} \chi_n}_{g_\Omega} \lesssim 1. 
\end{align}

Finally, let 
\begin{align*}
\phi_1 := \sum_m \hat{\chi}_m \cdot h_m
\end{align*}
and 
\begin{align*}
\phi_2 = \phi-\phi_1 =  \sum_m \hat{\chi}_m \cdot (\phi-h_m).
\end{align*} 

\begin{lemma} $\displaystyle \lim_{\zeta \rightarrow \partial \Omega} \int_{\Bb_\Omega(\zeta,r/2)} \abs{\phi_2}^2 dV_\Omega = 0$. In particular, 
\begin{enumerate}
\item ${\rm dom}(M_{\phi_2}) = A^2(\Omega)$ and $M_{\phi_2} : A^2(\Omega) \rightarrow L^2(\Omega)$ is a compact operator,
\item ${\rm dom}(H_{\phi_2}) = A^2(\Omega)$ and $H_{\phi_2} : A^2(\Omega) \rightarrow L^2(\Omega)$ is a compact operator,
\item ${\rm dom}(H_{\phi_1}) = {\rm dom}(H_\phi)$. 
\end{enumerate}
 \end{lemma}

\begin{proof} For the main assertion, it is enough to show that 
\begin{align*}
 \int_{\Bb_\Omega(\zeta,r/2)} \abs{\phi_2}^2 dV_\Omega\lesssim \max\{ \epsilon_m^2 : \zeta \in {\rm supp}(\chi_m)\}.
\end{align*}
Fix $\zeta \in \Omega$ and let 
\begin{align*}
\{ m_1,\dots, m_k\} &= \left\{ m :  {\rm supp}(\chi_m) \cap \Bb_{\Omega}(\zeta, r/2) \neq \emptyset \right\} \\
& \subset \{ m : \zeta_m \in \Bb_\Omega(\zeta, 3r/2)\}.
\end{align*}
Notice that $k \leq L$ and 
\begin{align*}
& \left(\int_{\Bb_\Omega(\zeta,r/2)} \abs{\phi_2}^2 dV_\Omega\right)^{1/2} = \left(\int_{\Bb_\Omega(\zeta,r/2)} \abs{\sum_{j=1}^k \hat{\chi}_{m_j}(\phi- h_{m_j})}^2 dV_\Omega\right)^{1/2} \\
& \quad \leq \sum_{j=1}^k\left( \int_{\Bb_\Omega(\zeta_{m_j},r)} \abs{\phi- h_{m_j}}^2 dV_\Omega\right)^{1/2} \leq L \max\{ \epsilon_m : \zeta \in {\rm supp}(\chi_m)\}.
\end{align*}

Next we prove the ``in particular'' assertions. Proposition~\ref{prop:multiplication_operators} immediately implies (1). Since $H_{\phi_2} = (\id - P_\Omega) \circ M_{\phi_2}$ and $\id -P_\Omega$ is a bounded operator, we see that (1) implies (2). Finally, since ${\rm dom}(M_{\phi_2})=A^2(\Omega)$, we see that  
$$
{\rm dom}(H_{\phi_1}) ={\rm dom}(M_{\phi_1}) ={\rm dom}(M_{\phi}) = {\rm dom}(H_\phi).
$$
\end{proof} 

Fix $r_2 <\frac{r}{2}$ sufficiently small such that: if $w \in {\rm supp}(\chi)$, then $\Bb(w, C_1 r_2) \subset \frac{r}{C_1}\Bb$. 

\begin{lemma}\label{lem:L2_diff_upper} If $\zeta \in {\rm supp}(\chi_n) \cap {\rm supp}(\chi_m)$, then 
\begin{align*}
\left(\int_{\Bb_\Omega(\zeta, r_2)}  \abs{h_n-h_m}^2 dV_\Omega\right)^{1/2} \leq \epsilon_n + \epsilon_m. 
\end{align*}
\end{lemma}

\begin{proof} 
If $\zeta \in {\rm supp}(\chi_n) \cap {\rm supp}(\chi_m)$, then
\begin{align*}
\Bb_\Omega(\zeta, r_2) \subset \Bb_\Omega(\zeta_n, r) \cap \Bb_\Omega(\zeta_m, r).
\end{align*}
So 
\begin{align*}
\left(\int_{\Bb_\Omega(\zeta, r_2)}  \abs{h_n-h_m}^2 dV_\Omega\right)^{1/2} & \leq \left(\int_{\Bb_\Omega(\zeta, r_2)} \abs{h_n-f}^2 dV_\Omega\right)^{1/2}+ \left(\int_{\Bb_\Omega(\zeta, r_2)} \abs{h_m-f}^2 dV_\Omega\right)^{1/2} \\
& \leq \epsilon_n + \epsilon_m. 
\end{align*}
\end{proof}

\begin{lemma} $\lim_{\zeta \rightarrow \partial \Omega} \int_{\Bb_\Omega(\zeta, r_2)} \norm{\bar{\partial} \phi_1}_{g_\Omega}^2 dV_\Omega = 0$. In particular, $H_{\phi_1}$ extends to a compact operator on $A^2(\Omega)$.  \end{lemma}

\begin{proof} To prove the first assertion, it is enough to show that 
\begin{align*}
\int_{\Bb_\Omega(\zeta, r_2)} \norm{\bar{\partial} \phi_1}_{g_\Omega}^2 dV_\Omega\lesssim \max\{ \epsilon_m^2 : \zeta \in {\rm supp}(\chi_m)\}.
\end{align*}

Fix $\zeta \in \Omega$ and let 
\begin{align*}
\{ m_1,\dots, m_k\} &= \left\{ m :  {\rm supp}(\chi_m) \cap \Bb_{\Omega}(\zeta, r_2) \neq \emptyset \right\} \\
& \subset \{ m : \zeta_m \in \Bb_\Omega(\zeta, r+r_2)\}.
\end{align*}
Notice that $k \leq L$ since $r_2 < \frac{r}{2}$. Also 
\begin{align*}
\bar{\partial} \phi_1(\zeta) = \sum_{j=1}^k h_{m_j}  \bar{\partial} \hat{\chi}_{m_j}
\end{align*}
on $\Bb_\Omega(\zeta, r_2)$. Further, since $\{ \hat{\chi}_m\}$ is a partition of unity, $\sum_{j=1}^k  \bar{\partial} \hat{\chi}_{m_j}= 0$ on $\Bb_\Omega(\zeta, r_2)$. So
\begin{align*}
\bar{\partial} \phi_1=\sum_{j=2}^k \left( h_{m_j}-h_{m_1}\right)  \bar{\partial} \hat{\chi}_{m_j}
\end{align*}
on $\Bb_\Omega(\zeta, r_2)$. Then by Equation~\eqref{eq:gradient_estimate} and Lemma~\ref{lem:L2_diff_upper}
\begin{align*}
\left( \int_{\Bb_\Omega(\zeta, r_2)} \norm{\bar{\partial} \phi_1}_{g_\Omega}^2 dV_\Omega\right)^{1/2} 
&\lesssim  \sum_{j=2}^k  \left(\int_{\Bb_\Omega(\zeta, r_2)} \abs{h_{m_j}-h_{m_1}}^2 dV_\Omega \right)^{1/2} \leq \sum_{j=2}^k (\epsilon_{m_j}+\epsilon_{m_1})  \\
& \lesssim \max\{ \epsilon_m : \zeta \in {\rm supp}(\chi_m)\}.
\end{align*}
This proves the first assertion. 

From Lemma~\ref{lem:L2_diff_upper} we know that ${\rm dom}(H_{\phi_1}) = {\rm dom}(H_\phi)$ and so $\phi_1 \in \Sc(\Omega)$. Hence Proposition~\ref{prop:smooth_symbols}  implies that $H_{\phi_1}$ extends to a compact operator. 
\end{proof} 

\begin{lemma} $H_\phi$ extends to a compact operator. \end{lemma}

\begin{proof} By the last two lemmas we see that $H_{\phi} = H_{\phi_1}+H_{\phi_2}$ extends to a compact operator on $A^2(\Omega)$. 
\end{proof}

\section{Proof of Theorem~\ref{thm:main_2}}\label{sec:pf_of_thm_main2}

The proof of Theorem~\ref{thm:main_2} is very similar to the proof of Theorem~\ref{thm:main} and is left to the reader.

\section{Proof of Theorem~\ref{thm:main_C0}}\label{sec:pf_of_thm_C0}

Suppose $\Omega \subset \Cb^d$ is a bounded domain with bounded intrinsic geometry, $\partial \Omega$ is $\Cc^0$, and $\phi \in \Cc(\overline{\Omega})$. Let $C_1 > 1$ and $\{ \Phi_\zeta : \zeta \in \Omega\}$ satisfy Theorem~\ref{thm:good_charts}.

Theorem~\ref{thm:main_C0} is a consequence of Theorem~\ref{thm:main} and the next three lemmas. 

\begin{lemma} If $\phi$ is holomorphic on every analytic variety in $\partial\Omega$, then there exists $r > 0$ such that 
$$
\lim_{\zeta \rightarrow \partial \Omega} \inf\left\{ \int_{\Bb_\Omega(\zeta,r)} \abs{\phi-h}^2  dV_\Omega  : h \in {\rm Hol}\left(\Bb_\Omega(\zeta,r)\right)\right\} = 0.
$$ 
\end{lemma} 

\begin{proof}  Fix $r < \frac{1}{2C_1}$. Then $\Bb_\Omega(\zeta, r) \subset \Phi_\zeta\left(\frac{1}{2}\Bb\right)$ for all $\zeta \in \Omega$.  Fix a sequence $(\zeta_m)_{m \geq 1}$ in $\Omega$ such that $\zeta_m \rightarrow \partial \Omega$ and
\begin{align*}
\limsup_{\zeta \rightarrow \partial \Omega} & \, \inf\left\{ \int_{\Bb_\Omega(\zeta,r)} \abs{\phi-h}^2  dV_\Omega  : h \in {\rm Hol}\left(\Bb_\Omega(\zeta,r)\right)\right\}\\
& =\lim_{m \rightarrow \infty} \inf\left\{ \int_{\Bb_\Omega(\zeta_m,r)} \abs{\phi-h}^2  dV_\Omega  : h \in {\rm Hol}\left(\Bb_\Omega(\zeta_m,r)\right)\right\}.
\end{align*}

Passing to a subsequence we can suppose that $\Phi_{\zeta_m}$ converges locally uniformly to $\Phi : \Bb \rightarrow \overline{\Omega}$. Since the Bergman distance on $\Omega$ is proper (by the Hopf-Rinow theorem) and $\Phi_\zeta(\Bb) \subset \Bb_\Omega(\zeta, C_1)$, we must have $\Phi(\Bb) \subset \partial\Omega$. Then by assumption, $h_0 := \phi \circ \Phi : \Bb \rightarrow \Cb$ is holomorphic. Then, if $h_m := \left. h_0 \circ \Phi^{-1}_{\zeta_m}\right|_{\Bb_\Omega(\zeta_m,r)}$ we have
\begin{align*}
\lim_{m \rightarrow \infty} & \int_{\Bb_\Omega(\zeta_m,r)} \abs{\phi-h_m}^2  dV_\Omega  \leq \limsup_{m \rightarrow \infty} \int_{\frac{1}{2}\Bb} \abs{\phi \circ \Phi_{\zeta_m}-h_0}^2  \Phi_{\zeta_m}^*dV_\Omega \\
& \leq C_1^{2d} \limsup_{m \rightarrow \infty} \int_{\frac{1}{2}\Bb} \abs{\phi \circ \Phi_{\zeta_m}-h_0}^2 d\mu =0.
\end{align*}

\end{proof}

To the prove the other direction in Theorem~\ref{thm:main_C0}, we first show that every analytic variety in $\partial \Omega$ can be locally obtained by taking a limit of the embeddings $\Phi_\zeta : \Bb \rightarrow \Omega$ 

\begin{lemma}\label{lem:holo_varieties_in_bd} If $F : \Db \rightarrow \partial \Omega$ is holomorphic and $z_0 \in \Db$, then there exist $\delta_0 > 0$ and a sequence $(\zeta_m)_{m \geq 1}$ such that $\Phi_{\zeta_m}$ converges locally uniformly to a holomorphic map $\Phi : \Bb \rightarrow \partial \Omega$ with $\Phi(0)=F(z_0)$ and
\begin{align*}
F(\Db(z_0, \delta)) \subset \Phi(\Bb).
\end{align*}
\end{lemma}

\begin{proof}Since $\partial \Omega$ is $\Cc^0$ there exists a unit vector $\nu \in \Cb^d$ and $\delta_0 > 0$ such that 
\begin{align*}
t \nu + F( \Db(z_0, \delta_0)) \subset \Omega
\end{align*}
for all $t \in (0,\delta_0)$. 

Let $\zeta_m := \frac{\delta_0}{m} \nu + F(z_0)$. Then let $d_\Omega^K$ and $d_{\Db}^K$ denote the Kobayashi distances on $\Omega$ and $\Db$ respectively. By the distance decreasing property of the Kobayashi metric, 
\begin{align*}
d_\Omega^K\left(\zeta_m, \frac{\delta_0}{m}\nu + F(w) \right) \leq d_{\Db}^K\left(0, \frac{w-z_0}{\delta_0} \right)
\end{align*}
for all $w \in \Db(z_0, \delta_0)$. Then by Theorem~\ref{thm:kob_vs_bergman}, there exists $\delta > 0$ such that 
\begin{align}
\label{eqn:inclusion}
 \frac{\delta_0}{m} \nu + \Psi( \Db(z_0, \delta)) \subset \Bb_\Omega\left(\zeta_m, \frac{1}{2C_1}\right) \subset \Phi_{\zeta_m}\left(\frac{1}{2}\Bb\right).
\end{align}
Passing to a subsequence we can suppose that $\Phi_{\zeta_m}$ converges locally uniformly to $\Phi : \Bb \rightarrow \overline{\Omega}$. Then
\begin{align*}
\Phi(0)=\lim_{m \rightarrow \infty} \Phi_m(0) = F(z_0).
\end{align*}
Also, since the Bergman distance on $\Omega$ is proper (by the Hopf-Rinow theorem) and $\Phi_\zeta(\Bb) \subset \Bb_\Omega(\zeta, C_1)$, we must have $\Phi(\Bb) \subset \partial\Omega$. Finally, Equation~\eqref{eqn:inclusion} implies that $F(\Db(z_0, \delta)) \subset \Phi(\Bb)$. 

\end{proof}

\begin{lemma} If there exists $r > 0$ such that $$\lim_{\zeta \rightarrow \partial \Omega} \inf\left\{ \int_{\Bb_\Omega(\zeta,r)} \abs{\phi-h}^2  dV_\Omega  : h \in {\rm Hol}\left(\Bb_\Omega(\zeta,r)\right)\right\} = 0,$$ 
then $\phi$ is holomorphic on every analytic variety in $\partial\Omega$.
\end{lemma}

\begin{proof} Using Lemma~\ref{lem:holo_varieties_in_bd}, it is enough to fix a sequence $(\zeta_m)_{m \geq 1}$ in $\Omega$ where $\zeta_m \rightarrow \partial \Omega$ and $\Phi_{\zeta_m}$ converges locally uniformly to a holomorphic map $\Phi : \Bb \rightarrow \partial \Omega$, then show that $\phi \circ \Phi$ is holomorphic in a neighborhood of $0$. 

By hypothesis, for each $m \geq 1$, there is some $h_m \in  {\rm Hol} \left( \Bb_\Omega(\zeta_m,r)\right)$ such that: if 
\begin{align*}
\epsilon_m: =\left(\int_{\Bb_\Omega(\zeta_m, r)} \abs{ \phi - h_m }^2  dV_\Omega\right)^{1/2}, 
\end{align*}
then $\lim_{m \rightarrow \infty} \epsilon_m = 0$. 

Fix $r_1 <\min\left\{ \frac{r}{C_1},1\right\}$. Then $\Phi_{\zeta_m}(r_1 \Bb) \subset \Bb_\Omega(\zeta_m,r)$ and so $\hat{h}_m : = h_m \circ \Phi_{\zeta_m}$ is well defined on $r_1 \Bb$. Also, if $\hat{\phi}_m := \phi \circ \Phi_{\zeta_m}$, then $\hat{\phi}_m$ converges uniformly on $r_1 \Bb$ to $\hat{\phi}:= \phi \circ \Phi$. 

By Theorem~\ref{thm:good_charts},
\begin{align*}
\int_{r_1 \Bb} \abs{\hat{\phi}_m - \hat{h}_m}^2 d\mu & \leq C_1^{2d} \int_{r_1 \Bb} \abs{\hat{\phi}_m - \hat{h}_m}^2 \Phi_{\zeta_m}^*dV_\Omega \\
& = C_1^{2d} \int_{\Phi_{\zeta_m}(r_1\Bb)}  \abs{ \phi - h_m }^2  dV_\Omega \leq C_1^{2d} \epsilon_m^2. 
\end{align*}
Since $\hat{\phi}_m$ converges uniformly to $\hat{\phi}$, we then see that $\int_{r_1 \Bb} \abs{\hat{h}_m}^2 d\mu$ is uniformly bounded. So after passing to a subsequence we can suppose that $\hat{h}_m$ converges locally uniformly to a holomorphic function $\hat{h}$ on $r_1 \Bb$. Then by Fatou's lemma
\begin{align*}
\int_{r_1 \Bb} \abs{\hat{\phi} - \hat{h}}^2 d\mu \leq \liminf_{m \rightarrow \infty} \int_{r_1 \Bb} \abs{\hat{\phi}_m - \hat{h}_m}^2 d\mu = 0. 
\end{align*}
So $\phi \circ \Phi=\hat{\phi}$ coincides with $\hat{h}$, a holomorphic function, on $r_1 \Bb$. 

\end{proof}

\section{When the standard potential has self bounded gradient}\label{sec:char_SBG}

In this section we characterize when the standard potential for the Bergman metric has self bounded gradient. 

\begin{theorem}\label{thm:char_self_bd_gradient} Suppose $\Omega \subset \Cb^d$ is a domain with bounded intrinsic geometry and $\{ \Phi_\zeta : \zeta \in \Omega\}$ satisfies Theorem~\ref{thm:good_charts}. Then the following are equivalent: 
\begin{enumerate}
\item $\log \Bf_\Omega(z,z)$ has self bounded gradient, 
\item For every $r > 0$  there exists $C=C(r) > 1$ such that: if $\dist_\Omega(z,\zeta) \leq r$, then 
\begin{align*}
\frac{1}{C} \leq \frac{\Bf_\Omega(z,z)}{\Bf_\Omega(\zeta,\zeta)} \leq C.
\end{align*}
\item For every $r > 0$  there exists $C=C(r) > 1$ such that: if $\zeta \in \Omega$, then 
\begin{align*}
 \frac{1}{C} \frac{1}{\mu\left( \Bb_\Omega(\zeta,r) \right)} \leq \Bf_\Omega(\zeta,\zeta) \leq C \frac{1}{\mu\left( \Bb_\Omega(\zeta,r) \right)}.
\end{align*}
\item For every $r > 0$  there exists $C=C(r) > 1$ such that: if $\zeta \in \Omega$, then 
\begin{align*}
 \frac{1}{C} \frac{d\mu}{\mu\left( \Bb_\Omega(\zeta,r) \right)} \leq dV_\Omega \leq C \frac{d\mu}{\mu\left( \Bb_\Omega(\zeta,r) \right)} 
\end{align*}
on $\Bb_\Omega(\zeta,r)$. 
\item 
\begin{align*}
\sup_{\zeta \in \Omega} \norm{ \left. \partial_w \log \abs{\det \Phi_\zeta^\prime(w)} \, \right|_{w=0}}_{2}< +\infty.
\end{align*}
\end{enumerate} 

\end{theorem}

The rest of the section is devoted to the proof of the theorem.

\begin{lemma} (1) $\Rightarrow$ (2). \end{lemma}

\begin{proof} Let
\begin{align*}
Q:=\sup_{z \in \Omega} \norm{\partial \log \Bf_\Omega(z,z)}_{g_\Omega} <+\infty.
\end{align*}
Since $\log \Bf_\Omega(z,z)$ is real valued, $\bar{\partial} \log \Bf_\Omega(z,z) = \overline{\partial \log \Bf_\Omega(z,z)}$. Hence 
\begin{align*}
\sup_{z \in \Omega} \norm{d \log \Bf_\Omega(z,z)}_{g_\Omega} \leq 2Q
\end{align*}
and so 
\begin{align*}
e^{-2Q \dist_\Omega(z,\zeta)}  \leq \frac{\Bf_\Omega(z,z)}{\Bf_\Omega(\zeta,\zeta)} \leq e^{2Q \dist_\Omega(z,\zeta)} 
\end{align*}
for all $z,\zeta \in \Omega$. 
\end{proof}

\begin{lemma} (2) $\Rightarrow$ (3).\end{lemma}

\begin{proof} Let $C_1 > 1$ be the constant from Theorem~\ref{thm:good_charts}. We consider two cases: 

\emph{Case 1:} Assume $r < \frac{1}{C_1}$. If $\zeta \in \Omega$, then Theorem~\ref{thm:good_charts} implies that 
\begin{align*}
\Phi_\zeta\left(\frac{r}{C_1} \Bb\right) \subset \Bb_\Omega(\zeta,r)  \subset \Phi_\zeta\left(C_1r \Bb \right).
\end{align*}
So
\begin{align*}
\int_{\frac{r}{C_1} \Bb} \abs{ \det \Phi_\zeta^\prime(w)}^2 d\mu(w) \leq \mu\left( \Bb_\Omega(\zeta,r) \right)  \leq \int_{C_1r \Bb} \abs{ \det \Phi_\zeta^\prime(w)}^2 d\mu(w).
\end{align*}
By Theorem~\ref{thm:Bergman_kernel_est} part (1) and the assumption 
\begin{align*}
\abs{ \det \Phi_\zeta^\prime(w)}^2  \asymp\frac{1}{\Bf_\Omega(\Phi_\zeta(w),\Phi_\zeta(w))} \asymp \frac{1}{\Bf_\Omega(\zeta,\zeta)}
\end{align*}
when $w \in \Bb$. So, in this case, there exists $C=C(r) > 1$ such that: if $\zeta \in \Omega$, then 
\begin{align*}
 \frac{1}{C} \frac{1}{\mu\left( \Bb_\Omega(\zeta,r) \right)} \leq \Bf_\Omega(\zeta,\zeta) \leq C \frac{1}{\mu\left( \Bb_\Omega(\zeta,r) \right)}.
\end{align*}

\emph{Case 2:} Assume $r \geq \frac{1}{C_1}$. Fix $r_0 < \frac{1}{C_1}$. By Corollary~\ref{cor:discretization} there exists a sequence $(\zeta_m)_{m \geq 1}$ of distinct points in $\Omega$ such that 
\begin{enumerate}
\item $\{\zeta_m : m \geq 1\}$ is $r_0$-separated with respect to the Bergman distance,
\item $\cup_m \Bb_\Omega(\zeta_m, r_0) = \Omega$, and 
 \item $L:=\sup_{z \in \Omega} \#\{ m : \zeta_m \in \Bb_\Omega(z, r+r_0) \} <+\infty$.
\end{enumerate}

Then if $\zeta \in \Omega$, case 1 implies that
 \begin{align*}
\mu\left( \Bb_\Omega(\zeta,r) \right) \geq \mu\left( \Bb_\Omega(\zeta,r_0) \right) \gtrsim \frac{1}{\Bf_\Omega(\zeta,\zeta)}.
 \end{align*}
Also, case 1 and the assumption imply that 
  \begin{align*}
\mu\left( \Bb_\Omega(\zeta,r) \right) \leq \sum_{\zeta_j \in \Bb_\Omega(\zeta,r+r_0)} \mu\left( \Bb_\Omega(\zeta_j,r_0) \right) \lesssim L  \max_{\zeta_j \in \Bb_\Omega(z,r+r_0)} \frac{1}{\Bf_\Omega(\zeta_j,\zeta_j)} \lesssim  \frac{1}{\Bf_\Omega(\zeta,\zeta)}
 \end{align*}
 (notice that the implicit constant depends on $r > 0$). So, in this case, there exists $C=C(r) > 1$ such that: if $\zeta \in \Omega$, then 
\begin{align*}
 \frac{1}{C} \frac{1}{\mu\left( \Bb_\Omega(\zeta,r) \right)} \leq \Bf_\Omega(\zeta,\zeta) \leq C \frac{1}{\mu\left( \Bb_\Omega(\zeta,r) \right)}.
\end{align*}

\end{proof}

\begin{lemma} (3) $\Rightarrow$ (2). \end{lemma}

\begin{proof} Fix $r > 0$. If $\dist_\Omega(z,\zeta) < r$, then  
\begin{align*}
\frac{\Bf_\Omega(z,z)}{\Bf_\Omega(\zeta,\zeta)} \lesssim \frac{\mu(\Bb_\Omega(\zeta,r))}{\mu(\Bb_\Omega(z,2r))} \leq 1
\end{align*}
and 
\begin{align*}
\frac{\Bf_\Omega(z,z)}{\Bf_\Omega(\zeta,\zeta)} \gtrsim \frac{\mu(\Bb_\Omega(\zeta,2r))}{\mu(\Bb_\Omega(z,r))} \geq 1
\end{align*}
 (notice that the implicit constants depend on $r > 0$). 
 \end{proof}
 
 \begin{lemma} (2 and 3) $\Rightarrow$ (4). \end{lemma}
 
 \begin{proof} Fix $r > 0$. By Proposition~\ref{prop:volume_comp}
\begin{align*}
dV_\Omega(z) \asymp \Bf_\Omega(z,z)d\mu(z). 
\end{align*}
So (2) and (3) imply that on $\Bb_\Omega(\zeta, r)$ we have 
\begin{align*}
dV_\Omega(z) \asymp \Bf_\Omega(\zeta,\zeta) d\mu(z) \asymp \frac{1}{\mu(\Bb_\Omega(\zeta,r))} d\mu(z)
\end{align*}
 (notice that the implicit constants depend on $r > 0$). 
 \end{proof}

\begin{lemma} (4) $\Rightarrow$ (2). \end{lemma}
\begin{proof} Fix $r > 0$. By Proposition~\ref{prop:volume_comp}
\begin{align*}
dV_\Omega(z) \asymp \Bf_\Omega(z,z)d\mu(z). 
\end{align*}
So for $z \in \Bb_\Omega(\zeta, r)$, we have 
\begin{align*}
\Bf_\Omega(z,z) \asymp \frac{1}{\mu(\Bb_\Omega(\zeta, r))} 
\end{align*}
and hence 
\begin{align*}
\Bf_\Omega(z,z) \asymp \Bf_\Omega(\zeta,\zeta)
\end{align*}
(notice that the implicit constants depend on $r > 0$).
\end{proof}

\begin{lemma} (2) $\Rightarrow$ (5). \end{lemma}

\begin{proof} Fix a sequence $(\zeta_m)_{m \geq 1}$ such that 
\begin{align*}
\sup_{\zeta \in \Omega} \norm{ \left. \partial_w \log \abs{\det \Phi_\zeta^\prime(w)} \, \right|_{w=0}}_{2} = \lim_{m \rightarrow \infty} \norm{ \left. \partial_w \log \abs{\det \Phi_{\zeta_m}^\prime(w)} \, \right|_{w=0}}_{2}.
\end{align*}
Define $f_m : \Bb \rightarrow \Cb$ by 
\begin{align*}
f_m(w) = \frac{\det \Phi_{\zeta_m}^\prime(w)}{\det \Phi_{\zeta_m}^\prime(0)}.
\end{align*}
By Theorem~\ref{thm:good_charts} there exists $C_1 > 1$ such that 
\begin{align*}
\Phi_\zeta(\Bb) \subset \Bb_\Omega(\zeta, C_1)
\end{align*}
for all $\zeta \in \Omega$. So by Theorem~\ref{thm:Bergman_kernel_est} part (1) and the assumption 
\begin{align*}
\abs{f_m(w)}^2 \asymp \frac{\Bf_\Omega(\zeta,\zeta)}{\Bf_\Omega(\Phi_\zeta(w), \Phi_\zeta(w))} \asymp 1. 
\end{align*}

Using Montel's theorem and passing to a subsequence we can suppose that $f_m$ converges locally uniformly to a holomorphic function $f : \Bb \rightarrow \Cb$. Then 
\begin{align*}
\sup_{\zeta \in \Omega} & \norm{ \left. \partial_w \log \abs{\det \Phi_\zeta^\prime(w)}\,\right|_{w=0} }_{2}  = \lim_{m \rightarrow \infty} \norm{ \left. \partial_w \log \abs{\det \Phi_{\zeta_m}^\prime(w)} \, \right|_{w=0}}_{2} \\
& = \lim_{m \rightarrow \infty}  \norm{\partial f_m(0)}_{2}= \norm{\partial f(0)}_{2} < +\infty.
\end{align*}
\end{proof}

\begin{lemma} (1) $\Leftrightarrow$ (5). \end{lemma}

\begin{proof} By Theorem~\ref{thm:good_charts}
\begin{align*}
\norm{ \left.\partial_z \log \Bf_\Omega(z,z) \right|_{z=\zeta}}_{g_\Omega} & = \norm{\left. \partial_w \log \Bf_\Omega( \Phi_\zeta(w), \Phi_\zeta(w)) \right|_{w=0}}_{\Phi_z^*g_\Omega} \\
& \asymp  \norm{\left. \partial_w \log \Bf_\Omega( \Phi_\zeta(w), \Phi_\zeta(w)) \right|_{w=0}}_{2}.
\end{align*}
Further, by Theorem~\ref{thm:Bergman_kernel_est}
\begin{align*}
 \norm{\left. \partial_w \log \beta_\zeta( w,w) \right|_{w=0}}_{2}
 \end{align*}
 is uniformly bounded and by definition 
\begin{align*}
\left. \partial_w \log \Bf_\Omega( \Phi_\zeta(w), \Phi_\zeta(w)) \right|_{w=0} =  \left. \left( \partial_w \log \beta_\zeta( w,w) - \partial_w \log \abs{\det \Phi_\zeta^\prime(w)}^2\right) \right|_{w=0}.
\end{align*}
So (5) $\Leftrightarrow$ (1).
\end{proof}

\bibliographystyle{alpha}
\bibliography{complex}

\end{document}